\documentclass[11pt, a4paper, reqno]{amsart}

\usepackage{verbatim, ifpdf}
  \usepackage{graphicx, color}

 \usepackage{amssymb, amsmath}
 \usepackage[a4paper,margin=30mm]{geometry} 

 \usepackage{amsthm}
 \usepackage{listings}
 \usepackage{bbm}
 \usepackage{mathtools}

 \usepackage{tikz-cd}

 \usepackage{soul}

 \usepackage{stmaryrd}

 \usepackage{enumitem}
 \usepackage{framed}

 \usepackage{mathrsfs}
 \usepackage{todonotes}

 \usetikzlibrary{decorations.markings}

\definecolor{colorForCitations}{HTML}{00643A}

  \IfFileExists{hyperref.sty}{%
    \ifpdf
      \usepackage[pdftex]{hyperref}
    \else
      \usepackage[dvips]{hyperref}
    \fi
    \definecolor{maroon}{RGB}{128,0,0}
    \hypersetup{pdfborder=0 0 0,colorlinks=true,plainpages=false,
     pageanchor=true,linkcolor=black,urlcolor=maroon,
     citecolor=colorForCitations}%
  }{%
    \providecommand{\href}[1]{\texttt{##1}}
    \providecommand{\href}[2]{##2\footnote{~See \texttt{##1}.}}%
  }

\newtheorem{theorem}{Theorem}[section]
\newtheorem{proposition}[theorem]{Proposition}
\newtheorem{lemma}[theorem]{Lemma}

\newtheorem{conjecture}[theorem]{Conjecture}
\newtheorem{corollary}[theorem]{Corollary}

\theoremstyle{remark}
\newtheorem{remark}[theorem]{Remark}
\theoremstyle{definition}

\newtheorem{definition}[theorem]{Definition}

\newcommand{\len}{\ell}

\newcommand{\ZZ}{\mathbb{Z}}
\newcommand{\QQ}{\mathbb{Q}}
\newcommand{\RR}{\mathbb{R}}
\newcommand{\CC}{\mathbb{C}}

\newcommand{\EE}{{\mathbb{E}}}
\newcommand{\PP}{{\mathbb{P}}}

\newcommand{\M}{\overline{\mathcal{M}}}
\newcommand{\D}{\widehat{\mathcal{D}}}
\newcommand{\lcm}{\mathrm{lcm}}
\newcommand{\denom}{\mathrm{denom}}
\newcommand{\Aut}{\mathrm{Aut}}
\newcommand{\id}{{\mathrm{id}}}

\newcommand{\Sp}{{\mathrm{Sp}}}
\newcommand{\FF}{{\mathbb{F}}}
\newcommand{\GL}{{\mathrm{GL}}}
\newcommand{\GG}{{\mathbb{G}}}
\newcommand{\Pic}{{\mathbf{Pic}}}

\begin{document}

\title[Denominators of Hodge integrals]{Denominators of Hodge integrals and Minkowski numbers}

\author[D. Holmes]{Daniel Holmes}
\address{Institute of Science and Technology Austria (ISTA)}
\email{daniel.holmes@ist.ac.at}
\urladdr{\href{https://www.daniel-holmes.at}{https://www.daniel-holmes.at}}

\subjclass[2020]{Primary 14H10, 
14C17; 
Secondary
14H37, 
11B68, 
11B83
}

\keywords{Hodge integrals, stable curves, Minkowski numbers, Bernoulli polynomials, N\o rlund polynomials, Gromov--Witten theory}

\begin{abstract}
We prove that the least common denominator of all genus $g$ Hodge integrals is the Minkowski number $T_{2g}$ or, equivalently, the Bhargava factorial $(2g+1)!_{\PP}$ associated to the set of primes $\PP\subset \mathbb{N}$.
As a corollary, we prove a conjecture of Liu--Xu on the multiplicity of $3$ in the least common denominators of certain descendant integrals, and generalise their result on a conjecture of Itzykson--Zuber to Hodge integrals with $\lambda$ classes.
Namely, the order of the automorphism group of any pointed stable curve of genus $g'\le g$ divides the least common denominator of all genus $g$ Hodge integrals.

We also comment on the use of generative AI via \textsf{IMProofBench}, which played a fruitful part in the development of this (entirely human-written) paper.
\end{abstract}

\maketitle


\section{Introduction}

Let $\M_{g,n}$ be the Deligne--Mumford moduli stack of stable curves of genus $g$ with $n$ marked points \cite{MR262240, MR702953}.
It carries several important vector bundles: the Hodge bundle $\EE$ of rank $g$, and for each $i\in\{1,\dots,n\}$ the cotangent line bundle $\mathcal{L}_i$.
The fibres over a stable curve $[C,p_1,\dots,p_n]$ are $H^0(C,\omega_C)$ and $T^*_{p_i}C$, respectively, where $\omega_C$ is the dualising sheaf of $C$
(see \cite[\S XVII.4]{MR2807457} for more details).

A \emph{Hodge integral} is a rational number of the form
\[
\int_{\M_{g,n}}\psi_1^{a_1}\dots\psi_n^{a_n}\lambda_1^{l_1}\dots\lambda_g^{l_g}
\]
where $\psi_i=c_1(\mathcal{L}_i)$, $\lambda_j=c_j(\EE)$, and the exponents are non-negative integers.
For $g>0$, Hodge integrals generally fail to be integers because $\M_{g,n}$ is not represented by a scheme, due to the presence of non-trivial automorphisms.

Motivation for studying Hodge integrals arises from intersection theory on the moduli of curves \cite{MR717614} and its tautological rings (see \cite{MR3821159} and references therein).
Via the Torelli map, this extends to the study of $\mathcal{A}_g$, the moduli space of principally polarised abelian varieties.
In Gromov--Witten theory, Hodge integrals appear as natural ingredients in the equivariant localisation formula \cite{MR1666787, MR1728879, MR1954265}.
They are related to Hurwitz numbers via the ELSV formula \cite{MR1864018} and to the representation theory of symmetric groups via the Gopakumar--Mari\~no--Vafa formula \cite{MR1950947, MR2057777, MR2199062}.

Given a rational number $a/b\in\QQ$ with coprime integers $a,b$ and $b$ positive, we write $\denom(a/b):=b$ for its denominator.
Consider the following integer, which is the main character of this paper.

\[
\D_g:=\lcm\left\{\denom \left(\int_{\M_{g,n}}\psi_1^{a_1}\dots\psi_n^{a_n}\lambda_1^{l_1}\dots\lambda_g^{l_g}\right) \,:\,\, \begin{matrix}
    n\in\ZZ_{\ge 0}\,,\,2g-2+n>0,\\
    a_1,\dots,a_n,l_1,\dots,l_g\in\ZZ_{\ge 0}
\end{matrix}\right\}.
\]

This is the least common denominator of all genus $g$ Hodge integrals, where the condition $2g-2+n>0$ ensures that $\M_{g,n}$ is non-empty.
This generalises the following integer studied by Itzykson--Zuber \cite{MR1180858} and Liu--Xu \cite{MR2595544} for $g\ge 2$.

\[
    \mathcal{D}_g := \lcm\left\{\denom \left(\int_{\M_{g,n}}\psi_1^{a_1}\dots\psi_n^{a_n}\right) : n\ge 1\,,\, a_i\ge 2\right\}.
\]

A key property of $\mathcal{D}_g$ is its relation to automorphism groups of Riemann surfaces, which was conjectured in \cite[end~of~\S5]{MR1180858} and proved in {\cite[Corollary~4.9]{MR2595544}}.

\begin{theorem}[Liu--Xu]\label{thm : Liu Xu auts}
    For any $1< g' \le g$, the order of the automorphism group of any compact Riemann surface of genus $g'$ divides $\mathcal{D}_g$.
\end{theorem}

Moreover, Liu--Xu obtained the following conjectural formula for $\mathcal{D}_g$ \cite[Conjecture~2.5]{MR2595544} and proved part (i) of it in the same paper. For any prime $p$, let $\nu_p\colon \QQ\rightarrow \ZZ\cup\{\infty\}$ be the $p$-adic valuation defined by $\nu_p(0) = \infty$ and $\nu_p(p^ka/b)=k$ for any $k\in\ZZ$ and $a,b\in\ZZ\setminus\{0\}$ coprime to $p$.

\begin{conjecture}[Liu--Xu]\label{conj : Liu Xu orders}
    Let $g\ge 2$ be an integer. Then
    \begin{itemize}
        \item[(i)] $\nu_2(\mathcal{D}_g) = 3g + \nu_2(g!)$
        \item[(ii)] $\nu_3(\mathcal{D}_g) = g + \nu_3(g!)$
        \item[(iii)] $\nu_p(\mathcal{D}_g) = \left\lfloor\frac{2g}{p-1}\right\rfloor$ for any prime $p\ge 5$.
    \end{itemize}
\end{conjecture}

As corollaries of our main result, we will generalise Theorem~\ref{thm : Liu Xu auts} to $\D_g$ and prove parts (i) and (ii) of Conjecture~\ref{conj : Liu Xu orders}.
The other protagonist of this paper is the integer sequence \cite[\href{https://oeis.org/A053657}{A053657}]{oeis} defined as follows.

\begin{definition}\label{def : Minkowski numbers}
    Let $m\in\mathbb{Z}_{\ge 0}$.
    The \emph{Minkowski number} $T_m$ is the positive integer defined such that for any prime $p$, the exponent $\nu_p(T_m)$ of $p$ in $T_m$ is
    \[
    \nu_p(T_m) = \sum_{i=0}^\infty \left\lfloor\frac{m}{(p-1)p^i}\right\rfloor.
    \]
\end{definition}
The first few values are $(T_m)_{m\ge 0} = (1, 2, 24, 48, 5760, 11520, 2903040,\dots)$.
Minkowski numbers have many beautiful connections with other areas of mathematics:
\begin{enumerate}[label=(M\arabic*),ref=(M\arabic*)]
    \item\label{item : minkowski prop 1} For any finite subgroup $G\subset \GL_m(\QQ)$, the order $|G|$ divides $T_m$ (see \cite{MR1580123, Schur} and cf. \cite[\S1]{MR2336645}).
    \item\label{item : minkowski prop 2}
    For any set $S\subset \ZZ$, Bhargava defines a generalised factorial function $k\mapsto k!_S$ on $\ZZ_{\ge 0}$, which recovers the ordinary factorial for $S=\ZZ$ \cite{MR1792411}.
    When $S=\PP$ is the set of primes, the Minkowski numbers are recovered via
    $T_m=(m+1)!_{\PP}$ (see \cite[Example~21]{MR1792411} and cf. \cite[(11)]{MR3493376}).
    \item\label{item : minkowski prop 3} $T_m$ is the least common denominator of the coefficients of $P_{m}(x)\in\QQ[x]$ in the expansion
    \[
    \left(\frac{t}{e^t-1}\right)^x = \sum_{m=0}^\infty P_m(x)t^m.
    \]
    Minkowski numbers are therefore related to Bernoulli polynomials, N\o rlund polynomials, and Stirling numbers (see \cite{MR883535, MR1691859} and cf. \cite[Proposition~2.1]{MR2265803}).
    \item\label{item : minkowski prop 4} $\{T_m\}_{m\ge 0}$ are the denominators of Laurent polynomials forming a regular basis for the Hopf algebroid of stable cooperations for complex $K$-theory \cite{MR730711, MR2434169}.
\end{enumerate}

Our main result is that we may add denominators of Hodge integrals to this list.

\begin{theorem}\label{thm : main thm}
    Let $g$ be any non-negative integer. Then the least common denominator of all genus $g$ Hodge integrals is the Minkowski number $T_{2g}$. That is, $\D_g=T_{2g}$.
\end{theorem}

As an immediate corollary, we obtain parts (i) and (ii) of Conjecture~\ref{conj : Liu Xu orders}.

\begin{corollary}\label{cor : conj LX i ii}
    For any $g\ge 2$, the following hold.
    \begin{enumerate}
        \item[(i)] $\nu_2(\mathcal{D}_g) = 3g+\nu_2(g!)$
        \item[(ii)] $\nu_3(\mathcal{D}_g) = g+\nu_3(g!)$
    \end{enumerate}
\end{corollary}
Here, the upper bounds on $\nu_2(\mathcal{D}_g)$ and $\nu_3(\mathcal{D}_g)$ follow from $\nu_p(\mathcal{D}_g)\le \nu_p(\D_g)$ and equations \eqref{eqn : v2 T2n} and \eqref{eqn : v3 T2n}, which state that the right-hand sides of (i) and (ii) are precisely $\nu_2(T_{2g})$ and $\nu_3(T_{2g})$, respectively.
The lower bounds follow from \eqref{eqn : 24^g g!}, which states that $24^g g!$ is the denominator of an appropriate $\psi$-class integral in genus $g$.

\begin{remark}
    Part (i) was proved differently in \cite[Theorem~4.3~(iv)]{MR2595544}
    by induction using the DVV recursion.
\end{remark}

Another immediate consequence is the generalisation of \cite[Theorem~4.3]{MR2595544} to Hodge integrals with $\lambda$ classes.
It follows immediately from Theorem~\ref{thm : main thm} and Definition~\ref{def : Minkowski numbers}.
\begin{corollary}\label{cor : 2g+1}
    Let $p$ be a prime number.
    Then
    \begin{enumerate}
        \item[(i)] If $p>2g+1$ then $p\nmid \D_g$.
        \item[(ii)] For any prime $p\le 2g+1$, we have $p\mid \D_g$.
        \item[(iii)] If $2g+1$ is prime then $\nu_{2g+1}(\D_g)=1$.
    \end{enumerate}
\end{corollary}
\begin{remark}
    For $g\ge 2$, part (i) was proved differently in \cite[Theorem~4.6]{MR2595544} using Mumford's formula for $\text{ch}(\EE)$ \cite[\S5]{MR717614} and Faber's algorithm \cite{MR1714822}.
    Similarly, part (ii) also follows from \cite[Theorem~4.3~(ii)]{MR2595544} and $\mathcal{D}_g\mid\D_g$.
\end{remark}

Finally, we obtain a generalisation of Theorem~\ref{thm : Liu Xu auts}.
Liu--Xu observed that its statement does not hold when \emph{Riemann surface} is replaced by \emph{stable curve}.
Our result shows that this is fixed if one simultaneously replaces $\mathcal{D}_g$ by $\D_g$.

\begin{corollary}\label{cor: aut orders divide Dg}
    For $0\le g'\le g$ and an algebraically closed field $k$ of arbitrary characteristic, the order of the automorphism group of any pointed stable curve over $k$ of genus $g'$ divides $\D_g$.
\end{corollary}

We prove this as Theorem~\ref{thm : cor aut} in Section~\ref{sec : cor proofs}.
In fact, Liu--Xu \cite[\S 4]{MR2357474} introduced
\[
\mathcal{S}_g := \lcm\left\{|\Aut(C)| : C \text{ is a complex stable curve of genus }g\right\}
\]
for $g\ge 2$ and observed that
\begin{align}
    \nu_2(\mathcal{S}_g)&\ge 2g + \nu_2(g!),
    \label{eqn : S_g ineq at 2}
    \\
    \nu_p(\mathcal{S}_g)
    &\ge k+\nu_p(k!),\,\,\, p\ge 3 \text{ prime},\,\,\,k=\left\lfloor\frac{2g}{p-1}\right\rfloor
    \label{eqn : S_g ineq at p}.
\end{align}
They conjectured that these inequalities are in fact equalities.
Since $\nu_p(T_{2g})=k+\nu_p(k!)$ (cf. Lemma~\ref{lem : nu_T_eg rewrite}), Theorem~\ref{thm : main thm} and Corollary~\ref{cor: aut orders divide Dg} imply the conjectured equality in \eqref{eqn : S_g ineq at p}.
In particular, $\mathcal{S}_g$ and $T_{2g}$ coincide up to a power of $2$.

\begin{remark}
    A different proof that $\mathcal{S}_g$ divides $T_{2g}$ appeared in the recent independent work \cite[\S4.2]{Philip2026}, which moreover discusses stable curves $C$ maximising $\nu_p(|\Aut(C)|)$.
    In particular, it seems that equality in \eqref{eqn : S_g ineq at 2} depends on the characteristic of the base field used to define $\mathcal{S}_g$.
\end{remark}

\subsection{Proof strategy}

While our proof strategy for Theorem~\ref{thm : main thm} uses analogues of properties \ref{item : minkowski prop 1} and \ref{item : minkowski prop 3} of the Minkowski numbers, it would be interesting to see a different proof using properties \ref{item : minkowski prop 2} or \ref{item : minkowski prop 4} instead.
Our strategy is as follows.
\begin{enumerate}
    \item[(i)] In Section \ref{sec : upper bound}, we prove $\D_g \mid T_{2g}$ (Corollary~\ref{cor : upper bound}) using compactified Galois covers of $\mathcal{M}_{g,n}$ with Galois group $\Sp_{2g}(\ZZ/N\ZZ)$, coming from the abelian level structure of order $N$ (Proposition~\ref{prop : geometry prop}).
    Varying $N$ and taking the greatest common divisor of $|\Sp_{2g}(\FF_q)|$ for primes $q\ge 3$ gives rise to $T_{2g}$ (Lemma~\ref{lem : gcd sp}).

    \item[(ii)] In Section \ref{sec : lower bound}, we show $T_{2g}\mid \D_g$ (Corollary~\ref{cor : T_2g divides D_g}) by exhibiting a suitable family of Hodge integrals.
    Specifically, we use from \cite[Theorem~2]{MR1728879} that
    \[
    \left(\frac{t/2}{\sin(t/2)}\right)^{x+1} 
    =
    1 + \sum_{g=1}^\infty t^{2g} \sum_{i=0}^g x^i\int_{\M_{g,1}}\psi_1^{2g-2+i}\lambda_{g-i}.
    \]
    Thus, it suffices to prove that the least common denominator of the coefficients of $Q_{2g}(x)\in\QQ[x]$ in the expansion
    \[
    \left(\frac{t/2}{\sin(t/2)}\right)^x = \sum_{g=0}^\infty Q_{2g}(x) t^{2g}
    \]
    is precisely $T_{2g}$ (see Proposition~\ref{prop : T_2g Q_2g primitive}).
\end{enumerate}
Thus, Theorem~\ref{thm : main thm} follows immediately from (i) and (ii).

\subsection{Statement on the use of generative AI}

This paper is entirely human-written.
The following paragraphs concern the human/AI interaction involved in the development of the underlying ideas.

This project grew out of a numerical observation made during the development of \cite{holmes} and the associated software package \textsf{GKMtools} \cite{GKMtools}.
The author submitted resulting conjectures to \textsf{IMProofBench} \cite{schmitt2025improofbenchbenchmarkingairesearchlevel}, a platform to benchmark current leading AI systems against research-level mathematical questions.
While most models produced technically wrong answers, some were wrong in an interesting way that helped to refine the conjecture and identify potential proof methods.

In summary, the strategy used in \S\ref{sec : divisibility lemmas} was correctly provided by OpenAI's \textsf{GPT-5.4 Pro} in \textit{extended} mode. The approach of \S\ref{sec : upper bound} using level structures was suggested via \textsf{IMProofBench} by the following models: \textsf{GPT-5.2}, \textsf{GPT-5.2 Pro}, \textsf{Gemini-3.1 Pro Preview}, \textsf{Claude Opus 4.6}, and \textsf{GPT-5}.
Although very useful, it still took non-trivial human effort to fill in the technical gaps and mistakes contained in those attempts.
In particular, there was a previous stage of the human/AI interaction at which all leading models in \textsf{IMProofBench} uniformly suggested a superficially good-looking approach for \S2 that was technically bound to fail.
On the other hand, a previous version of the conjecture wrongly proposed $\D_g$ to coincide with the sequence \cite[\href{https://oeis.org/A277001}{A277001}]{oeis} of denominators appearing in Luschny's approximation to the gamma function (see \cite[Example~6.1]{MR3459588}).
This was pointed out via \textsf{IMProofBench} by \textsf{Gemini-3.1 Pro Preview}.

\subsection{Acknowledgements}
I thank Johannes Schmitt for introducing me to \textsf{IMProofBench} during my visit to ETH Z{\"u}rich from November 2025 to January 2026, Rahul Pandharipande for hosting me there, and my advisor Tam{\'a}s Hausel for many helpful discussions and for funding this visit.
Further thanks are due to Marco Boggi for providing useful references on abelian level structures; Jakub L{\"o}wit, Johannes Schmitt, Shiyu Shen, Tanguy Vernet, and Angelo Vistoli for providing references and discussion on the intersection theory of stacks; and Mischa Elkner, Giosu{\`e} Muratore, and Yannik Schuler for comments that helped to improve this draft.
Finally, I thank Carel Faber for suggesting to extend Corollary~\ref{cor: aut orders divide Dg} to fields of arbitrary characteristic.
This work was supported by the ERC grant ViaFiPoS, No 101199663 and, from September 2026, by a DOC Fellowship of the Austrian Academy of Sciences.

\section{An upper bound using level structures}
\label{sec : upper bound}

In this section, we prove the upper bound $\D_g\mid T_{2g}$.

We shall use the notion of \emph{level structures} for $\mathcal{M}_{g,n}$, which in the unpointed case go back at least to \cite{SerreJacobi1960, MR262240}, cf. \cite[\S XVI]{MR2807457}.
The general theme here is the search for Galois covers of $\mathcal{M}_{g,n}$, preferably by schemes \cite{MR1257324, MR1739177}.
To extend these to covers of $\M_{g,n}$, there are at least two different methods in the literature.
The first one, which we use, is by normalisation (cf. \cite{MR717614, MR801921}).
This produces a cover of $\M_{g,n}$ by a projective variety, which has the pleasant feature that degrees of integrally defined $0$-cycles on it are integers.
The second one, which we mention for completeness, is by extending the notion of level structures to stable (possibly nodal) curves in suitable ways and considering the resulting moduli functors \cite{MR2007376}.

We shall also use Chow groups with rational coefficients for Deligne--Mumford stacks, as developed in \cite{MR1005008}.
For a separated Deligne--Mumford stack $\mathcal{X}$ of finite type over $\mathbb{C}$, let $A_*(\mathcal{X})_\QQ$ be its Chow group with rational coefficients and let $A^*(\mathcal{X})_\QQ$ be its Chow cohomology ring.
For a scheme $X$
of finite type over $\CC$, we write $A_*(X)$ and $A^*(X)$ for its Chow group and Chow cohomology ring with integer coefficients, respectively 
\cite[\S1-3,\,\S17.1-3]{MR1644323}.
Although we will not need it, let us mention that an intersection theory for Deligne--Mumford stacks with integer coefficients was developed in \cite{MR1614555, MR1719823}, see also \cite[Appendices~B,C]{MR4430955}.
The resulting integral Chow groups differ only by torsion and agree on algebraic spaces \cite[Theorem~2.1.12~(i),(ii)]{MR1719823}.

\begin{proposition}\label{prop : geometry prop}
    For any $N\ge 3$, we have $\D_g \mid |\Sp_{2g}(\ZZ/N\ZZ)|$.
\end{proposition}
\begin{proof}
    A \emph{level $N$ structure} on a smooth stable curve $C$ of genus $g$ over $\CC$ is an isomorphism
    \[
    H^1(C,\ZZ/N\ZZ)\stackrel{\cong}{\longrightarrow}(\ZZ/N\ZZ)^{2g}
    \]
    of symplectic $\ZZ/N\ZZ$-modules, where the symplectic structure is given by the cup product on the left and by a fixed standard symplectic structure on the right.
    More generally, if $S$ is a scheme over $\CC$ and $C\rightarrow S$ is a smooth stable curve of genus $g$ over $S$, then a \emph{level $N$ structure} of $C\rightarrow S$ is a symplectic isomorphism $J[N]\stackrel{\cong}{\longrightarrow}(\ZZ/N\ZZ)^{2g}_S$ where $J[N]$ is the $N$-torsion of the relative Jacobian $J=\Pic_{C/S}^0$, endowed with its symplectic structure via the Weil pairing.

    Let $\mathcal{M}_{g,n}(N)$ be the moduli stack of smooth stable curves of genus $g$ with $n$ marked points, equipped with a level $N$-structure.
    Let $f\colon \mathcal{M}_{g,n}(N)\rightarrow \mathcal{M}_{g,n}$ be the forgetful morphism.

    Surveys of the following facts, including marked points, may be found in \cite[\S2]{MR2242960}, see also \cite[\S1]{MR1739177}.

    \begin{itemize}
        \item $f\colon \mathcal{M}_{g,n}(N)\rightarrow \mathcal{M}_{g,n}$ is a finite surjective representable {\'e}tale Galois cover of degree $D:=|\Sp_{2g}(\ZZ/N\ZZ)|$.
        \item Since $N\ge 3$, $\mathcal{M}_{g,n}(N)$ is an algebraic space \cite{SerreJacobi1960}, and in fact an algebraic variety (cf. \cite[(5.14)]{MR262240} and \cite[Proposition~2.4]{MR2242960}).
        \item Let $\M_{g,n}(N)$ be the normalisation of $\M_{g,n}$ in the function field of $\mathcal{M}_{g,n}(N)$, extending $f$ to $\bar{f}\colon \M_{g,n}(N)\rightarrow \M_{g,n}$.
        \item 
        Since $N\ge 3$, $\M_{g,n}(N)$ is a proper algebraic space \cite[\S3]{MR801921}, and in fact a projective variety (cf. \cite[Proposition~2.5]{MR2242960}).
    \end{itemize}

Hence we have the following commutative diagram, where $\bar{f}$ restricts to $f$ over $\mathcal{M}_{g,n}(N)$.
    \[
    \begin{tikzcd}
        \M_{g,n}(N) \ar["\bar f"]{r}\ar["\pi_2"']{dr}
        &
        \M_{g,n} \ar["\pi_1"]{d}
        \\
        &\text{Spec}(\CC)
    \end{tikzcd}
    \]
    Note that $\bar f$ is proper by \cite[\href{https://stacks.math.columbia.edu/tag/0CPT}{Lemma 0CPT}]{stacks-project} since $\M_{g,n}(N)\rightarrow\text{Spec}(\CC)$ is proper and $\M_{g,n}\rightarrow \text{Spec}(\CC)$ is separated.
    In the following, we will use the proper pushforward $\bar f_*\colon A_*(\M_{g,n}(N))_\QQ\rightarrow A_*(\M_{g,n})_\QQ$ and the projection formula
    $\bar f_*(\bar f^*(\alpha)\frown z) = \alpha\frown \bar f_*(z)$ for $\alpha\in A^*(\M_{g,n})_\QQ$ and $z\in A_*(\M_{g,n}(N))_\QQ$.
    
    Let $I:=\psi_1^{a_1}\dots\psi_n^{a_n}\lambda_1^{l_1}\dots\lambda_g^{l_g}\in A^*(\M_{g,n})_\QQ$ and note that we may assume \[\deg(I) = a_1+\dots+a_n+l_1+2l_2 +\dots+gl_g=\dim\M_{g,n}\] as otherwise $\int_{\M_{g,n}}I=0$.
    By the above hypotheses, we have 
    \[
        \bar f_*[\M_{g,n}(N)] = D[\M_{g,n}]\in A_{\dim\M_{g,n}}(\M_{g,n})_\QQ.
    \]
    Hence, the Hodge integral is given by
    \begin{align}\label{eqn : key eqn}
    \int_{\M_{g,n}}I &= (\pi_1)_*(I\frown [\M_{g,n}])\nonumber
    \\&= \frac{1}{D}(\pi_1)_*(I\frown \bar f_*[\M_{g,n}(N)])\nonumber
    \\ &= \frac{1}{D}(\pi_1)_*\bar f_*(\bar f^*I\frown [\M_{g,n}(N)])\nonumber
    \\&= \frac{1}{D}(\pi_2)_*(\bar f^*I\frown [\M_{g,n}(N)])
    \in A_0(\text{Spec}(\CC))_\QQ.
    \end{align}

Note that the final expression in \eqref{eqn : key eqn} is computed purely scheme-theoretically.
    Indeed, $\bar f^*I$ is a product of Chern classes of vector bundles on the projective variety $\M_{g,n}(N)$ pulled back from $\M_{g,n}$.
    Since the degree of an integrally defined $0$-cycle on a proper scheme is an integer 
    \cite[Definition~1.4]{MR1644323}
    and the classes $\bar f^*I\in A^*(\M_{g,n}(N))$ and $[\M_{g,n}(N)]\in A_*(\M_{g,n}(N))$ are defined in the respective Chow groups with integer coefficients, we obtain
    \[
    (\pi_2)_*(\bar f^*I\frown [\M_{g,n}(N)])\in A_0(\text{Spec}(\CC))=\ZZ.
    \]
    Hence, by equation \eqref{eqn : key eqn},
    \[
    \int_{\M_{g,n}}I\in \frac{1}{D}\ZZ.
    \]
    Since the above argument holds for all $I$ and $n$, we obtain $\D_g \mid D=|\Sp_{2g}(\ZZ/N\ZZ)|$, as required.
\end{proof}

The following lemma may be found in \cite[Formula~3.1]{Silverberg1992}, where it is attributed to the methods of \cite{MR1580123}.
See also \cite[Theorem~6]{MR2336645} for groups of other types.

\begin{lemma}\label{lem : gcd sp}
    We have $\gcd \left\{|\Sp_{2g}(\FF_q)| : q\ge 3\text{ prime}\right \} = T_{2g}$.
\end{lemma}

Combining Proposition~\ref{prop : geometry prop} and Lemma~\ref{lem : gcd sp} gives the desired upper bound for $\D_g$.

\begin{corollary}\label{cor : upper bound}
    We have $\D_g\mid T_{2g}$.
\end{corollary}

\begin{remark}
    In this section, we used abelian level structures of order $N\ge 3$ and obtained an upper bound for $\D_g$ by computing the greatest common divisor of their degrees.
    Alternatively, one could attempt to use the finite surjective morphisms $\mathscr{B}^\mathrm{tei}_{g,n}(G_{\{p_1,p_2\}})\rightarrow \M_{g,n}$ from \cite[Theorem~7.6.4]{MR2007376}, where $\mathscr{B}^\mathrm{tei}_{g,n}(G_{\{p_1,p_2\}})$ is a \emph{smooth} projective variety defined in terms of distinct primes $p_1$ and $p_2$.
    However, more work is required to compute the degrees of these morphisms and their greatest common divisor.
    A similar remark applies to the construction in \cite{MR1257324}.
\end{remark}

\section{Linear Hodge integrals with one marked point}
\label{sec : lower bound}

In this section, we prove that $T_{2g}$ divides $\D_{g}$ by providing a lower bound for the exponent of each prime $p$ in $\D_g$.
For $p\in\{2,3\}$, this easily follows from the Hodge integral in equation~\eqref{eqn : 24^g g!}.
For $p\ge 5$, some more work is required, which is the content of Section~\ref{sec : divisibility lemmas}.
In Section~\ref{sec : FP} we derive the least common denominators appearing in the expansion of $\left(\frac{t/2}{\sin(t/2)}\right)^{x+1}$, which is reminiscent of property \ref{item : minkowski prop 3} of Minkowski numbers.
In Section~\ref{sec : conclusion} we arrive at the desired conclusion $T_{2g}\mid\D_{g}$.

\subsection{Preliminary divisibility lemmas}
\label{sec : divisibility lemmas}

Throughout this subsection, let $g\in\ZZ_{\ge 0}$ and let $p\ge 5$ be a prime.
For any $m\in\{0,\dots,g\}$, define
\begin{equation}\label{eqn : def Agm}
        A_{g,m} := (-1)^{g+m}\sum_{\substack{\lambda \vdash g\\\len(\lambda)=m}} \frac{1}{\prod_i \lambda_i!(2i+1)!^{\lambda_i}}
\end{equation}
where the sum runs over partitions $\lambda$ of $g$ of length $\ell(\lambda)=m$.
That is, $\lambda=(\lambda_1,\lambda_2,\dots\ge 0)$ with $\sum_{i} i\lambda_i = g$ and $\len(\lambda):=\sum_{i}\lambda_i=m$.
We also define the integers
\[
s := \frac{p-1}{2}\ge 2
\,\,\,\,\,\text{and}\,\,\,\,\,
q :=\left\lfloor\frac{2g}{p-1}\right\rfloor
\]
so that
\[
g = qs + r
\,\,\,\,\,\text{for some }r\in\ZZ\text{ with}\,\,\,\,\,
0\le r<s.
\]

The following two lemmas are well-known, but we record them for clarity.
\begin{lemma}[Legendre]\label{lem : Legendre}
    For $N\in\ZZ_{\ge 1}$ and a prime $p$, we have
    \[
        \nu_p(N!) = \sum_{i\ge 1}\left\lfloor\frac{N}{p^i}\right\rfloor = \frac{N - s_p(N)}{p-1}
    \]
    where $s_p(N)$ is the sum of the base $p$ digits of $N$.
    In particular,
    \[
    \nu_p(N!)\le\left\lfloor\frac{N-1}{p-1}\right\rfloor.
    \]
\end{lemma}

\begin{lemma}\label{lem : multinomial}
    Given integers $N$ and $\lambda_1,\lambda_2,\dots \ge 0$ with $\sum_{i\ge 1}\lambda_i \le N$, we have
    \[
    \sum_{i\ge 1}\nu_p(\lambda_i!) \le \nu_p(N!)
    \]
\end{lemma}
\begin{proof}
    The statement follows immediately from integrality of the multinomial coefficient $\binom{\sum_{i}\lambda_i}{\lambda_1,\lambda_2,\dots}$ and the fact that $\nu_p(N!)$ is non-decreasing in $N$.
\end{proof}

\begin{lemma}\label{lem : nu_T_eg rewrite}
    We have $\nu_p(T_{2g}) = q + \nu_p(q!)$.
\end{lemma}
\begin{proof}
    We have
    \[
    \nu_p(T_{2g})
    =
    \sum_{j\ge 0}\left\lfloor\frac{2g}{p^j (p-1)}\right\rfloor
    =
    q +
    \sum_{j\ge 1}\left\lfloor\frac{\frac{2g}{p-1}}{p^j}\right\rfloor
    =
    q +
    \sum_{j\ge 1}\left\lfloor\frac{q}{p^j}\right\rfloor
    =q+\nu_p(q!)
    \]
    where the second equality from the right uses the identity $\lfloor y / N\rfloor = \lfloor\lfloor y \rfloor / N\rfloor$ for $N\in\ZZ_{>0}$ and $y\in\RR$, and the rightmost equality is by Lemma~\ref{lem : Legendre}.
\end{proof}

\begin{lemma}\label{lem : lambda star divisible}
    Let $m:= q+r$ and let $\lambda^*=(\lambda_1,\lambda_2,\dots\ge 0)$ be defined by
    \[
    \lambda_1 = r,\,\,\,\,\,\,\,\,\,\,\lambda_s=q,\,\,\,\,\,\,\,\,\,\,\lambda_i=0\,\,(i\notin \{1,s\}).
    \]
    Then $\lambda^*\vdash g$, $\len(\lambda^*)=m$, and the denominator of the $\lambda^*$-summand of $A_{g,m}$ in \eqref{eqn : def Agm} is divisible by $p^{\nu_p(T_{2g})}$.
\end{lemma}
\begin{proof}
    First, $\lambda^*$ is well-defined since $p\ge 5$ implies $s\ge2$.
    Then $\lambda^*\vdash g$ and $\len(\lambda^*)=m$ hold by definition.
    The denominator of the $\lambda^*$-summand in \eqref{eqn : def Agm} is
    \begin{equation}\label{eqn : denom of lambda star}
    r! q! (3!)^{r}(2s+1)!^q
    =
    r! (3!)^rq! p!^q.
    \end{equation}
    Since $p\ge 5$ and $r<s<p$, we have $p\nmid 3!$ and $p\nmid r!$.
    Thus, the exponent of $p$ in \eqref{eqn : denom of lambda star} is
    \[
    \nu_p(q!) + q \nu_p(p!) = \nu_p(q!) + q = \nu_p(T_{2g}),
    \]
    as desired, where the right equation is Lemma~\ref{lem : nu_T_eg rewrite}.
\end{proof}

\begin{lemma}\label{lem : other lambda not divisible}
    Let $m:= q+r$ and let $\lambda=(\lambda_1,\lambda_2,\dots\ge 0)$ with $\lambda\vdash g$ and $\len(\lambda)=m$ be different from $\lambda^*$ in Lemma~\ref{lem : lambda star divisible}.
    Then $p^{\nu_p(T_{2g})}$ does \emph{not} divide the denominator of the $\lambda$-summand of $A_{g,m}$ in \eqref{eqn : def Agm}.
\end{lemma}
\begin{proof}
    The exponent of $p$ in the denominator of the $\lambda$-summand of $A_{g,m}$ in \eqref{eqn : def Agm} is
    \[
    D(\lambda):=\underbrace{\sum_{i\ge 1} \nu_p(\lambda_i!)}_{(\ast)} + \underbrace{\sum_{i\ge 1}\lambda_i \nu_p((2i+1)!)}_{(\ast\ast)}.
    \]
    We will show that $D(\lambda)< \nu_p(T_{2g})$.
    Define
    \[
    f:=\sum_{i\ge 1}\lambda_i\left\lfloor\frac{i}{s}\right\rfloor
    \,\,\,\,\,,\,\,\,\,\,
    t := q - f
    \]
    and note that $t\ge 0$ since
    \[
    f \le \frac{\sum_{i\ge 1}i\lambda_i}{s} = \frac{g}{s} = \frac{2g}{p-1}
    \]
    and $f\in\ZZ$, so
    \begin{equation}\label{eqn : f le q}
    f\le \left\lfloor \frac{2g}{p-1} \right\rfloor = q.
    \end{equation}
    For $(\ast\ast)$, Lemma~\ref{lem : Legendre} gives
    \[
    \nu_p((2i+1)!) \le \left\lfloor\frac{2i}{p-1}\right\rfloor = \left\lfloor\frac{i}{s}\right\rfloor
    \]
    so
    \begin{equation}\label{eqn : nu_p factorials}
        \sum_{i\ge 1} \lambda_i \nu_p((2i+1)!)\le f.
    \end{equation}
    For $(\ast)$, we split the sum into \emph{large} indices $i\ge s$ and \emph{small} indices $i<s$.
    For large indices, we have
    \begin{equation}\label{eqn : g ge s ineq}
    N_{\ge s} := \sum_{i\ge s}\lambda_i \le \sum_{i\ge 1}\lambda_i\left\lfloor\frac{i}{s}\right\rfloor = f
    \end{equation}
    and hence, by Lemma~\ref{lem : multinomial},
    \begin{equation}\label{eqn : nu_p large indices}
        \sum_{i\ge s}\nu_p(\lambda_i!) \le \nu_p(f!).
    \end{equation}
    
    For small indices, we use $f\le (\sum_{i\ge s}i\lambda_i)/s$ to get
    \[
    \sum_{i<s}i\lambda_i = g - \sum_{i\ge s}i\lambda_i \le g - sf
    = r + s(q-f) = r+st
    \]
    and hence
    \begin{equation}\label{eqn : N < s ineq}
    N_{<s} := \sum_{i < s}\lambda_i \le \sum_{i<s} i\lambda_i \le r+st.
    \end{equation}
    Thus, Lemma~\ref{lem : multinomial} gives
    \begin{equation}\label{eqn : nu_p small indices}
        \sum_{i<s}\nu_p(\lambda_i!) \le \nu_p((r+st)!).
    \end{equation}

    \textit{Claim: We have $t\ge 1$.}
    
    \textit{Proof of Claim.}
    Assume for contradiction that $t=0$.
    Then $f=q$ and \eqref{eqn : g ge s ineq}, \eqref{eqn : N < s ineq} give $N_{\ge s} \le q$ and $N_{<s} \le r$.
    But since
    \[
    N_{<s} = m -N_{\ge s}\ge m-q = r,
    \]
    we actually have $N_{<s}=r$ and hence $N_{\ge s}=q$.
    The former gives equality in \eqref{eqn : N < s ineq} which forces $\lambda_i=0$ for $1<i<s$ and hence $\lambda_1=r$.
    Thus, $\sum_{i\ge s}i \lambda_i=qs$, which together with $N_{\ge s}=q$ implies $\lambda_s=q$ and $\lambda_i=0$ for $i>s$.
    Thus, $\lambda=\lambda^*$, contradicting our assumption.
    We therefore have $t\ge 1$, proving the claim.

    \textit{Conclusion.}
    Since $r<s$ and $t\ge 1$, we have
    \[
    r + st < s + st \le 2st = (p-1)t
    \]
    so Lemma~\ref{lem : Legendre} gives
    \begin{equation}\label{lem : nu_p r+st}
        \nu_p((r+st)!)\le \left\lfloor\frac{(p-1)t-1}{p-1}\right\rfloor = t-1.
    \end{equation}
    Combining this with \eqref{eqn : nu_p factorials}, \eqref{eqn : nu_p large indices}, \eqref{eqn : nu_p small indices} gives
    \begin{align*}
    D(\lambda) &\le f + \nu_p(f!) + (t-1) 
    \\ &= q + \nu_p(f!) - 1 
    \\&\le q + \nu_p(q!) - 1
    \\ &= \nu_p(T_{2g})-1
    \end{align*}
    as desired, where the bottom inequality uses $f\le q$ from \eqref{eqn : f le q} and the final equality uses Lemma~\ref{lem : nu_T_eg rewrite}.
    This finishes the proof of the lemma.
\end{proof}

\begin{corollary}\label{cor : some Agm has divisible denominator}
    Let $m:= q+r$. Then $p^{\nu_p(T_{2g})}$ divides the denominator of $A_{g,m}$ in \eqref{eqn : def Agm}.
\end{corollary}
\begin{proof}
    By Lemma~\ref{lem : lambda star divisible} and Lemma~\ref{lem : other lambda not divisible}, $\lambda^*$ is the only partition whose summand has denominator divisible by $p^{\nu_p(T_{2g})}$.
    The statement now follows from the inequality $\nu_p(a+b)\ge \min\{\nu_p(a),\nu_p(b)\}$ for $a,b\in\QQ\setminus\{0\}$, which is an equality if $\nu_p(a)\neq \nu_p(b)$.
    Here we used that all summands in \eqref{eqn : def Agm} are non-zero.
\end{proof}

\begin{proposition}\label{prop : t over sin t}
    Let $p\ge 5$ be prime and let $R_{2g}(x)\in\QQ[x]$ be defined by
    \[
    \left(\frac{t}{\sin t}\right)^x = \sum_{g\ge 0} R_{2g}(x)t^{2g}.
    \]
    Then for each $g\in\ZZ_{\ge 0}$ there exists a coefficient $c\neq 0$ of $T_{2g}R_{2g}(x)$ with $\nu_p(c)\le 0$.
\end{proposition}

\begin{proof}
    For $m\ge 0$, write $x^{\overline{m}}:=x(x+1)\cdots(x+m-1)$ and recall the variant
    \[
    (1-z)^{-x} = \sum_{m\ge 0} \frac{x^{\overline{m}}}{m!}z^m
    \]
    of the generalised binomial theorem.
    Applying this to
    \[
    \left(\frac{t}{\sin t}\right)^{x}
    =
    \left(\frac{\sin t}{t}\right)^{-x}
    =
    \left(1 - \sum_{i\ge 1}(-1)^{i+1}\frac{t^{2i}}{(2i+1)!}\right)^{-x}
    \]
    we obtain
    \begin{equation}\label{eqn : Agm sum}
    R_{2g}(x) = \sum_{m=0}^g A_{g,m}x^{\overline{m}}
    \end{equation}
    with $A_{g,m}$ as defined in \eqref{eqn : def Agm}.
    Let
    \[
    e_g:=\max\left\{e\in\ZZ_{\ge 0} : p^e \mid \denom(A_{g,m})\,\,\text{for some}\,\,m\in\{0,\dots,g\}\right\}
    \]
    and
    \[
    m_1 := \max\left\{m\in\{0,\dots,g\} : p^{e_g} \mid \denom(A_{g,m})\right\}.
    \]

    Let $c'\in\QQ$ be the coefficient of $x^{m_1}$ in $R_{2g}(x)$.
    By definition of $e_g$ and $m_1$, and since $x^{\overline{m_1}}$ is monic, we have $p^{e_g}\mid \denom(c')$.
    In particular, $c'\neq 0$ whenever $e_g>0$.
    For $e_g=0$, we similarly obtain $m_1=g$ and hence $c'=A_{g,g}=1/g!(3!)^g\neq 0$.
    By Corollary~\ref{cor : some Agm has divisible denominator}, we have $0\le \nu_p(T_{2g})\le e_g$, so $c:=T_{2g}c'$ is the desired coefficient of $T_{2g}R_{2g}(x)$.

\end{proof}

\subsection{Faber--Pandharipande's Hodge integrals}
\label{sec : FP}

Recall from \cite[Theorem~2]{MR1728879} that
\begin{equation}\label{eqn : FP}
    \left(\frac{t/2}{\sin(t/2)}\right)^{x+1} 
    =
    1 + \sum_{g=1}^\infty t^{2g} \sum_{i=0}^g x^i\int_{\M_{g,1}}\psi_1^{2g-2+i}\lambda_{g-i}.
\end{equation}
Thus, writing
\begin{equation}\label{eqn : def Qn}
\left(\frac{t/2}{\sin(t/2)}\right)^x = \sum_{g \ge 0}Q_{2g}(x)t^{2g}
\end{equation}
with $Q_{2g}(x)\in\QQ[x]$ as in the introduction, it suffices to prove the following.

\begin{proposition}\label{prop : T_2g Q_2g primitive}
    For any $g\in \ZZ_{\ge 0}$, the polynomial $T_{2g} Q_{2g}(x)$ has integer coefficients and is primitive.
\end{proposition}
\begin{proof}
    The case $g=0$ is immediate, so assume $g\ge 1$.

    First, note that the shift $x\mapsto x+1$ does not change the content of a polynomial.
    So by equations \eqref{eqn : FP}, \eqref{eqn : def Qn} and $\D_g\mid T_{2g}$ (Corollary~\ref{cor : upper bound}), $T_{2g}Q_{2g}(x)$ has integer coefficients.
    Thus, it suffices to find for every prime $p$ a coefficient of $T_{2g}Q_{2g}(x)$ that is not a multiple of $p$.

    For $p\in\{2,3\}$, the leading coefficient works.
    Indeed, by \cite[after~(2.26)]{MR1144529}, the Witten conjecture \cite[Theorem~1.2]{MR1171758} implies we have the following (see also \cite[(28)]{MR1728879}):
    \begin{equation}\label{eqn : 24^g g!}
    \int_{\M_{g,1}}\psi_1^{3g-2} = \frac{1}{24^g g!}.
    \end{equation}
    In particular,
    \begin{equation}\label{eqn : v2 T2n}
    \nu_2(T_{2g}) = \sum_{i\ge 0}\left\lfloor \frac{2g}{2^i} \right\rfloor
    =
    2g + g + \sum_{i\ge 1}\left\lfloor\frac{g}{2^i}\right\rfloor 
    =3g+\nu_2(g!)
    =
    \nu_2(24^g g!)
    \end{equation}
    and
    \begin{equation}\label{eqn : v3 T2n}
    \nu_3(T_{2g}) = \sum_{i\ge 0}\left\lfloor\frac{2g}{2\cdot 3^i}\right\rfloor
    = g + \sum_{i\ge 1}\left\lfloor\frac{g}{3^i}\right\rfloor
    =
    g+\nu_3(g!)
    =
    \nu_3(24^g g!)
    \end{equation}
    so the coefficient of $x^g$ in $T_{2g}Q_{2g}(x)$ is coprime to 2 and 3.

    Finally, the remaining case $p\ge 5$ was treated in Proposition~\ref{prop : t over sin t} above.
    For notational convenience, we substituted $t/2$ by $t$ in that statement.
    This is valid because we have already dealt with the exponents of $p=2$ above, and the exponents of all other primes remain unchanged.
\end{proof}

\subsection{Conclusion}
\label{sec : conclusion}

\begin{corollary}\label{cor : T_2g divides D_g}
    For any $g\ge 0$, we have $T_{2g}\mid \D_g$.
\end{corollary}
\begin{proof}
    This is an immediate consequence of equations \eqref{eqn : FP}, \eqref{eqn : def Qn} and Proposition~\ref{prop : T_2g Q_2g primitive}, noting that the shift $x\mapsto x+1$ does not change the content of a polynomial in $\QQ[x]$.
\end{proof}

\section{Proof of the automorphism corollary}\label{sec : cor proofs}

Finally, let us present the proof of Corollary~\ref{cor: aut orders divide Dg}, which we restate here for convenience.
Whenever $X$ is an algebraic group over a field $k$, we write $\Aut_{k\text{-gp}}(X)$ for the group of algebraic group automorphisms of $X$ over $k$.

\begin{theorem}\label{thm : cor aut}
    For $0\le g'\le g$ and an algebraically closed field $k$ of arbitrary characteristic, the order of the automorphism group of any pointed stable curve over $k$ of genus $g'$ divides $\D_g$.
\end{theorem}
\begin{proof}
    Let $(C,p_1,\dots,p_n)$ be a stable curve over $k$ of genus $g'$
    and let $\Gamma$ be its dual graph.

Let $\Pic^0(C)$ denote the generalised Jacobian of $C$, that is, the group of invertible sheaves of degree zero on every irreducible component of $C$.
    By \cite[Theorem~(1.13)]{MR262240},
    \[
    A:=\Aut_k(C,p_1,\dots,p_n)\longrightarrow \Aut_{k\text{-gp}}(\Pic^0(C))
    \]
    is injective. (In loc.\ cit., the unpointed case with $g'\ge 2$ is proved, while the general statement can be deduced by attaching a smooth genus 2 curve to each $p_i$ and extending any automorphism of $(C,p_1,\dots,p_n)$ trivially to those components.)
    Recall that the normalisation exact sequence for $C$ gives a short exact sequence
    \begin{equation*}
    0\longrightarrow \GG_m^{b_1(\Gamma)}
    \longrightarrow \Pic^0(C)
    \longrightarrow
    \prod_{v\in V(\Gamma)}\Pic^0(\widetilde{C_v})
    \longrightarrow 0,
    \end{equation*}
    where $\widetilde{C_v}$ is the normalisation of the component $C_v\subset C$ corresponding to the vertex $v\in V(\Gamma)$.
    Hence, $\Pic^0(C)$ is a semiabelian variety of dimension $b_1(\Gamma)+\sum_{v\in V(\Gamma)}g(\widetilde{C_v}) = g'$.
    Lemma~\ref{lem : auts of semi-abelian} below therefore implies $|A| \mid T_{2g'}$.
    Since $g'\le g$, Definition~\ref{def : Minkowski numbers} immediately gives $T_{2g'}\mid T_{2g}$.
    Since $\D_g=T_{2g}$ by Theorem~\ref{thm : main thm}, we obtain $|A|\mid \D_g$, as desired.
\end{proof}

It remains to provide the following lemma, which should be well-known.
Similar arguments may be found in \cite[Lemme~3.5.1]{MR801921} and, in a slightly different context, \cite[\S5--6]{SilverbergZarhin1998}.

\begin{lemma}\label{lem : auts of semi-abelian}
    Let $k$ be an algebraically closed field of arbitrary characteristic and let $S$ be a semiabelian variety over $k$, that is, $S$ is a commutative group variety over $k$ and there exists a short exact sequence
    \begin{equation}\label{eqn : semi-abelian SES}
    0\longrightarrow T\longrightarrow S\longrightarrow E\longrightarrow 0
    \end{equation}
    of algebraic groups over $k$ where $T$ is a torus and $E$ is an abelian variety.
    Let $A\subset\Aut_{k\text{-gp}}(S)$ be a finite subgroup.
    Then $|A|$ divides $T_{\dim(T)}T_{2\dim(E)}$ and hence $T_{2\dim(S)}$.
\end{lemma}
\begin{proof}
    By Chevalley's theorem (cf.\ \cite{Conrad2002Chevalley}), $T$ is the unique smooth connected closed affine normal subgroup of $S$ such that $S/T$ is an abelian variety.
    Therefore, $\Aut_{k\text{-gp}}(S)$ preserves $T$ and acts by automorphisms on $E$, making the short exact sequence in \eqref{eqn : semi-abelian SES} $\Aut_{k\text{-gp}}(S)$-equivariant.
    In particular, if $A_T$ and $A_E$ are the images of $A$ in $\Aut_{k\text{-gp}}(T)$ and $\Aut_{k\text{-gp}}(E)$, respectively, we obtain a map
    \[
    A\longrightarrow A_T\times A_E
    \]
    which we claim to be injective.
    Indeed, assume that $a\in A$ acts trivially on $T$ and $E$ and consider the map $f_a\colon S\rightarrow S$ sending $s$ to $a(s)s^{-1}$.
    It is a homomorphism since $S$ is commutative.
    Since $a$ acts trivially on $E$, $f_a$ takes values in $T$.
    Since $a$ acts trivially on $T$, $f_a$ is trivial on $T$.
    Thus, $f_a$ factors as
    \[
     S\longrightarrow E \stackrel{f_a'}{\longrightarrow} T \longrightarrow S.
    \]
    Since $E$ is proper and $T$ is affine, $f_a'$ is constant.
    Hence, $f_a$ is trivial and so $a=\id_S$, showing the desired injectivity $A\hookrightarrow A_T\times A_E$.
    In particular, we have
    \[
    |A| \mid |A_T||A_E|.
    \]

    By \cite[Lemma~8.2]{RosslerSchroer2026}, we have $|A_E|\mid T_{2\dim(E)}$.
    
    By property~\ref{item : minkowski prop 1} of the Minkowski numbers, we have $|A_T|\mid T_{\dim(T)}$ since $A_T$ is a finite subgroup of $\Aut_{k\text{-gp}}(T)\cong \GL_{\dim(T)}(\ZZ)$.
    Therefore, we get
    \[
    |A| \mid T_{\dim(T)}T_{2\dim(E)}
    \]
    as desired.
    The final statement that $T_{\dim(T)}T_{2\dim(E)}\mid T_{2\dim(S)}$ follows from the more general fact that $T_b\mid T_{b+1}$ and $T_b T_c\mid T_{b+c}$ for any $b,c\in \ZZ_{\ge 0}$, which is easily checked from Definition~\ref{def : Minkowski numbers}.
\end{proof}

\bibliographystyle{amsalpha}
\bibliography{refs}

\end{document}